\documentclass[11pt,a4article]{amsart}
\usepackage{url}
\usepackage{amsmath,amsthm,amsfonts,amssymb,latexsym}

\newtheorem{theorem}{Theorem}[section]

\newtheorem{lemma}[theorem]{Lemma}

\theoremstyle{definition}
\newtheorem{definition}[theorem]{Definition}

\newtheorem{question}[theorem]{Question}

\newcommand{\U}{\mathcal U}
\newcommand{\w}{\omega}

\newcommand{\IP}{\mathbb P}

\newcommand{\V}{\mathcal{V}}

\newcommand{\bigvid}{\hat{\ \ }}

\newcommand{\uhr}{\upharpoonright}

\newcommand{\name}[1]{\dot{#1}}
\newcommand{\la}{\langle}
\newcommand{\ra}{\rangle}

\newcommand{\forces}{\Vdash}

\newcommand{\hot}{\mathfrak}

\newcommand{\nothing}[1]{}

\title[$H$-separable spaces in the Laver model]{Products of $H$-separable spaces in the Laver model}

\author{Du\v{s}an Repov\v{s}  and Lyubomyr Zdomskyy} 

\address{Faculty of Mathematics and Physics,
University of Ljubljana, Jadranska 19, SI-1000 Ljubljana, Slovenia.}
\email{dusan.repovs@guest.arnes.si}
\urladdr{http://www.fmf.uni-lj.si/\~{}repovs/index.htm}

\address{Kurt G\"odel Research Center for Mathematical Logic,
University of Vienna, W\"ahringer Stra\ss e 25, A-1090 Wien,
Austria.}
\email{lzdomsky@gmail.com}
\urladdr{http://www.logic.univie.ac.at/\~{}lzdomsky/}


\thanks{The first author was  supported by 
the Slovenian Research Agency grants P1-0292, J1-8131, J1-7025, J1-6721, and J1-5435.
The second author would
like to thank  the Austrian Science Fund FWF (Grants I 1209-N25 and I 2374-N35)
 for generous support of this research.}

\begin{document}

\begin{abstract}
We prove that in the Laver model
 for the consistency of the Borel's conjecture,
the product of any two $H$-separable spaces is $M$-separable.
\end{abstract}

\maketitle

\section{Introduction}

This paper is devoted to products of $H$-separable spaces.
A topological space $X$ is said \cite{BelBonMat09}  to be \emph{$H$-separable}, if for
every sequence $\la D_n:n\in\w\ra$ of dense subsets of $X$, one can
pick finite subsets $F_n\subset D_n$ so that every nonempty open set
$O\subset X$ meets all but finitely many $F_n$'s. If 
we only demand that $\bigcup_{n\in\w}F_n$ is dense we get the definition
of \emph{$M$-separable} spaces introduced in \cite{Sch99}. 
It is obvious that second-countable spaces 
(even spaces with a countable $\pi$-base) are $H$-separable,
and each $H$-separable space is $M$-separable. 
The main result of our paper is the following

\begin{theorem}\label{main}
In the Laver model for the consistency of the Borel's conjecture,
the product of any two countable $H$-separable  spaces
is $M$-separable. 

Consequently,
the product of any two $H$-separable  spaces
is $M$-separable provided that it is hereditarily separable.
\end{theorem}

It worth mentioning here that by \cite[Theorem~1.2]{RepZdo10} the 
 equality $\hot b=\hot c$ which holds in the Laver model implies 
that the $M$-separability is not preserved by finite products of countable
spaces in the strong sense.
  
Let us recall that a topological space
$X$ is said to have the  \emph{Menger property} (or, alternatively, is a \emph{Menger space})
 if for every sequence $\la \U_n : n\in\omega\ra$
of open covers of $X$ there exists a sequence $\la \V_n : n\in\omega \ra$ such that
each $\V_n$ is a finite subfamily of $\U_n$ and the collection $\{\cup \V_n:n\in\omega\}$
is a cover of $X$. This property was introduced by  Hurewicz, and the current name
(the Menger property) is used because Hurewicz
proved in   \cite{Hur25} that for metrizable spaces his property is equivalent to
a certain  property of a base considered by Menger in \cite{Men24}.
If in the definition above we additionally require that  $\{n\in\w:x\not\in\cup\V_n\}$ 
is finite for each $x\in X$,
then we obtain the definition of the \emph{Hurewicz  property}  introduced
in \cite{Hur27}. The original idea behind the Menger's   property,
as it is explicitly stated in the first paragraph of \cite{Men24},  
was an application in dimension theory,
 one of the areas of
interest of Marde\v{s}i\'c.
However,  this paper 
concentrates on  set-theoretic and combinatorial aspects of the property of Menger
and its variations.

Theorem~\ref{main} is closely related to the main result of \cite{RepZdo??}
asserting that in the Laver model the product of any
two Hurewicz metrizable spaces has the Menger property. Let us note that 
our proof in \cite{RepZdo??} is 
 conceptually 
different, even though both proofs are based on  the same main technical lemma of \cite{Lav76}.
Regarding the relation between Theorem~\ref{main} and the main result of \cite{RepZdo??},
each of them implies a weak form of the other one via
the following duality results:
For a metrizable space $X$, $C_p(X)$ is $M$-separable 
(resp. $H$-separable) if and only if all finite powers of $X$ are Menger (resp. Hurewicz),
see \cite[Theorem~35]{Sch99} and \cite[Theorem~40]{BelBonMat09}, respectively.
 Thus Theorem~\ref{main} (combined
with the well-known fact that $C_p(X)$ is hereditarily separable for metrizable separable spaces
$X$) implies that in the Laver model, if all finite powers of metrizable separable
 spaces $X_0, X_1$
are Hurewicz, then $X_0\times X_1$ is Menger. And vice versa: The main result of
\cite{RepZdo??} implies that in the Laver model,
the product of two $H$-separable spaces of the form $C_p(X)$ for a metrizable separable
$X$, is   $M$-separable. 

The proof of Theorem~\ref{main}, which is  
 based on the analysis of names for reals in the  style of \cite{Lav76}, 
 unfortunately seems   to be rather tailored for the $H$-separability and we were 
not able to prove any
analogous results even for small variations thereof. Recall from \cite{GruSak11} that a 
space $X$ is said to be \emph{$wH$-separable} if for any \emph{decreasing}  sequence 
$\la D_n:n\in\w\ra$ of dense subsets of $X$, one can pick
finite subsets $F_n\subset D_n$ such that for any non-empty open $U\subset X$ 
the set $\{n\in\w:U\cap F_n\neq\emptyset\}$ is co-finite. It is clear that every $H$-separable
space is $wH$-separable, and it seems to be unknown whether the converse is
(at least consistently) true. Combining \cite[Lemma~2.7(2) and Corollary~4.2]{GruSak11}
we obtain that every countable Fr\'echet-Urysohn space is $wH$-separable,  
and to our best knowledge it is open whether  countable Fr\'echet-Urysohn spaces must be $H$-separable. 
The statement  
 ``finite products of countable Fr\'echet-Urysohn spaces are $M$-separable''  is known to be independent from
ZFC: It follows from the PFA by \cite[Theorem~3.3]{BarDow12}, holds in the Cohen model
by \cite[Corollary~3.2]{BarDow12}, and fails under CH by \cite[Theorem~2.24]{BarDow11}.
Moreover\footnote{We do not know whether the spaces constructed in the proof of \cite[Theorem~2.24]{BarDow11}
are $H$-separable.}, CH implies the existence of two countable Fr\'echet-Urysohn $H$-separable
topological groups whose product is not $M$-separable, see \cite[Corollary~6.2]{MilTsaZso16}. 
These results motivate the following 
\begin{question}
\begin{enumerate}
 \item Is it consistent that the  product of two countable $wH$-separable spaces is $M$-separable?
Does this statement hold in the Laver model?
\item Is the product of two countable Fr\'echet-Urysohn space $M$-separable in the Laver model?
\item Is the product of three (finitely many) countable $H$-separable spaces $M$-separable in the 
Laver model?
\item  Is the product of finitely many countable $H$-separable spaces $H$-separable in the 
Laver model?
\end{enumerate}
\end{question}

\section{Proof of Theorem~\ref{main}}

We need  the following

\begin{definition}
 A topological space $\la X,\tau\ra$ is called \emph{box-separable}
if for every  function $R$
assigning to each countable family  $\U$ of non-empty open subsets of $X$ 
a sequence $R(\U)=\la F_n:n\in\w\ra$ of finite non-empty subsets of $X$
such that $\{n:F_n\subset U\}$ is infinite for every $U\in\U$,
there exists $\mathsf U  \subset [\tau\setminus\{\emptyset\}]^{\w}$ 
of size $|\mathsf U|=\w_1$ such that for all $U\in\tau\setminus\{\emptyset\}$
there exists $\U\in\mathsf U$ such that
$\{n:R(\U)(n)\subset U\}$ is infinite.
\end{definition}

Any countable space is obviously box-separable  under CH,
which makes the latter notion uninteresting when  considered in arbitrary ZFC models.
However,  as we shall see in Lemma~\ref{covering_g_delta}, the box-separability becomes useful
under $\hot b>\w_1$. Here  $\hot b$ denotes the minimal cardinality of
a subspace $X$ of $\w^\w$ which is not eventually dominated by a single
function, see \cite{Bla10} for more information on $\hot b$ and other cardinal characteristics
of the reals. 

The following lemma is the key part of the proof of Theorem~\ref{main}.
 We will use the notation from \cite{Lav76} with  the only difference being  that
smaller conditions in a forcing poset  are supposed to
 carry more information about the generic filter, and the ground model is denoted by $V$.

A subset $C$ of $\w_2$ is called an
\emph{$\w_1$-club} if it is unbounded and for every $\alpha\in\w_2$ of cofinality $\w_1$,
if $C\cap\alpha$ is cofinal in $\alpha$ then $\alpha\in C$.

\begin{lemma} \label{laver}
 In the Laver model every countable $H$-separable  space is box-separable.
\end{lemma}
\begin{proof}
We work in $V[G_{\w_2}]$, where $G_{\w_2}$ is $\IP_{\w_2}$-generic
and $\IP_{\w_2}$ is the iteration of length $\w_2$ with countable supports
of the Laver forcing, see \cite{Lav76} for details.
 Let us fix an $H$-separable  space of the form $\la\w,\tau\ra$
and a function $R$ such as in the definition of  box-separability.
By a standard argument (see, e.g., the proof of \cite[Lemma~5.10]{BlaShe87})
there exists an $\w_1$-club $C\subset \w_2$ such that for every $\alpha\in C$ 
the following conditions hold:
\begin{itemize}
 \item[$(i)$] $\tau\cap V[G_\alpha] \in V[G_\alpha]$   and
for every sequence $\la D_n:n\in\w\ra\in V[G_\alpha]$ of dense subsets of
$\la \w,\tau\ra$ there exists a sequence  $\la K_n:n\in\w\ra\in V[G_\alpha]$
such that $K_n\in [D_n]^{<\w}$ and for every $U\in\tau\setminus\emptyset$
the intersection $U\cap K_n$ is non-empty for all but finitely many
$n\in\w$;
\item[$(ii)$] $R(\U)\in V[G_\alpha]$ for any $\U\in [\tau\setminus\{\emptyset\}]^\w\cap V[G_\alpha]$; and
\item[$(iii)$] For every $A\in\mathcal P(\w)\cap V[G_\alpha]$ the interior
$\mathit{Int}(A)$ also belongs to $V[G_\alpha]$.
\end{itemize}
By \cite[Lemma~11]{Lav76}
there is no loss of generality
in assuming that $0\in C$. We claim that $\mathsf U:=[\tau\setminus\{\emptyset\}]^\w\cap V$
is a witness for $\la\w,\tau\ra$ being box-separable. Suppose, contrary to our claim, that
there exists $A\in\tau\setminus\{\emptyset\}$ such that $R(\U)(n)\not\subset A$ for all but finitely many
$n\in\w$ and $\mathcal U\in\mathsf U$. Let $\name{A}$ be a $\IP_{\w_2}$-name for $A$ and $p\in\IP_{\w_2}$
 a condition forcing the above statement.
 Applying \cite[Lemma~14]{Lav76}
to the sequence $\la \name{a}_i:i\in\w\ra$ such that $\name{a}_i=\name{A}$ for all $i\in\w$,
we get a condition $p'\leq p$ such that $p'(0)\leq^0 p(0)$, and a finite set 
$\U_s\subset\mathcal P(\w)$
 for every $s\in p'(0)$ with $p'(0)\la 0\ra\leq s$, such that for each $n\in\w$,
  $s\in p'(0)$ with $p'(0)\la 0\ra\leq s$, and for all but finitely many
immediate successors $t$ of $s$ in $p'(0)$ we have
$$  p'(0)_t\bigvid p'\uhr[1,\w_2)\forces \exists U\in \U_s\: (\name{A}\cap n= U\cap n). $$
Of course, any $p''\leq p'$ also has the property above with the same 
$\U_s$'s. However, the stronger $p''$ is, the more elements 
of $\U_s$ might play no role any more. Therefore throughout the rest of the proof we
shall call $U\in\U_s$ \emph{void for $p''\leq p'$ and  $s\in p''(0)$}, where 
$p''(0)\la 0\ra\leq s$, if there exists $n\in\w$ such that for all but finitely many
immediate successors $t$ of $s$ in $p''(0)$ there is \emph{no} 
$q\leq p''(0)_t\bigvid p''\uhr[1,\w_2)$
with the property $q\forces \name{A}\cap n= U\cap n.$ 
Note that for any $p''\leq p'$ and $s\in p''(0)$,  
$p''(0)\la 0\ra\leq s$, there exists $U\in\U_s$ which is non-void for $p'',s$.
Two cases are possible.
\smallskip

$a)$  For every  $p''\leq p'$ there exists $s\in p''(0)$,  
$p''(0)\la 0\ra\leq s$, and a non-void $U\in\U_s$ for $p'',s$ such that
$\mathit{Int}(U)\neq\emptyset$. In this case let $\U\in\mathsf U$ be any countable family
containing $\{\mathit{Int}(U):U\in\bigcup_{s\in p'(0),p'(0)\la 0\ra\leq s}\U_s\}\setminus\{\emptyset\}$.
It follows from the above that $p$ forces 
$R(\U)(k)\not\subset \name{A}$ for all but finitely many
$k\in\w$. Let $p''\leq p'$ and $m\in\w$ be such that
$p''$ forces $R(\U)(k)\not\subset \name{A}$ for all $k\geq m$. 
Fix a non-void $U$ for $p'', s$, where $s\in p''(0)$ and  
$p''(0)\la 0\ra\leq s$, such that $\mathit{Int}(U)\neq\emptyset$ (and hence $\mathit{Int}(U)\in\U$).  
It follows from the above that there exists $k\geq m$ such that
$R(\U)(k)\subset \mathit{Int}(U)\subset U$. Let $n\in\w$ be such that
$R(\U)(k)\subset n$. By the definition of being non-void 
there are infinitely many
immediate successors $t$ of $s$ in $p''(0)$  for which there exists 
$q_t\leq p''(0)_t\bigvid p''\uhr[1,\w_2)$
with the property $q_t\forces \name{A}\cap n= U\cap n.$ Then for any $q_t$
as above we have that $q_t$ forces $R(\U)(k)\subset \name{A}$
because  $R(\U)(k)\subset U\cap n$, which contradicts the fact that
$q_t\leq p''$ and $p''\forces R(\U)(k)\not\subset \name{A} $.
\smallskip

$b)$  There exists   $p''\leq p'$ such that for all  $s\in p''(0)$,  
$p''(0)\la 0\ra\leq s$, every $U\in\U_s$ with $\mathit{Int}(U)\neq\emptyset$
 is void for $p'',s$. Note that this implies that every $U\in\U_s$ with $\mathit{Int}(U)\neq\emptyset$
 is void for $q,s$ for all $q\leq p''$ and $s\in q(0)$ such that  
$q(0)\la 0\ra\leq s$.

 Let $\la D_k:k\in\w\ra\in V$ be a sequence 
of dense subsets of $\la\w,\tau\ra$ such that for every 
$U\in\bigcup_{s\in p''(0),p''(0)\la 0\ra\leq s}\U_s$, if $\mathit{Int}(U)=\emptyset$,
then $\w\setminus U=D_k$ for infinitely many $k\in\w$.
Let $\la K_k:k\in\w\ra\in V$ be such as in item $(i)$ above.
Then $p''$ forces that $K_k\cap \name{A}\neq\emptyset$ for all but finitely many $k\in\w$.
Passing to a stronger condition, we may additionally assume  if necessary, that
there exists $m\in\w$ such that $p''\forces \forall k\geq m\:(K_k\cap \name{A}\neq\emptyset)$.

Fix $U\in\U_{p''(0)\la 0\ra}$ non-void for $p'', p''(0)\la 0\ra$. 
Then   $\mathit{Int}(U)=\emptyset$ by the choice of $p''$ and hence there exists 
$k\geq m$ such that $\w\setminus U=D_k$. It follows that
$K_k\cap U=\emptyset$ because $K_k\subset D_k$. On the other hand,
since $U$ is non-void for $p'', p''(0)\la 0\ra$, for $n=\max K_k+1$
we can find infinitely many 
immediate successors $t$ of $p''(0)\la 0\ra$ in $p''(0)$ for which there exists 
$q_t\leq p''(0)_t\bigvid p''\uhr[1,\w_2)$
forcing $\name{A}\cap n= U\cap n.$ Then any such $q_t$ forces $K_k\cap\name{A}=\emptyset$
(because $K_k\subset n$ and $K_k\cap U=\emptyset$),
contradicting the fact that $p''\geq q_t$ and $p''\forces K_k\cap\name{A}\neq\emptyset$.
\smallskip

Contradictions obtained in cases $a)$ and $b)$ above imply that
$\mathsf U:=[\tau\setminus\{\emptyset\}]^\w\cap V$
is a witness for $\la\w,\tau\ra$ being box-separable, which completes our proof. 
\end{proof}

Theorem~\ref{main} is a direct consequence of Lemma~\ref{laver} combined with the following

\begin{lemma} \label{covering_g_delta}
Suppose that $\hot b>\w_1$, $X$ is box-separable,  and $Y$ is $H$-separable.
Then $X\times Y$ is $M$-separable provided that it is separable.
\end{lemma}
\begin{proof}
Let $\la D_n:n\in\w\ra$ be a sequence of countable dense subsets of $X\times Y$. 
Let us fix a countable family $\U$ of open non-empty subsets of 
$X$ and a partition $\w=\sqcup_{U\in\U}\Omega_U$ into infinite pieces.
For every $n\in \Omega_U$ set $D^\U_n=\{y\in Y:\exists x\in U (\la x,y\ra\in D_n)\}$
and note that $D^\U_n$ is dense in $Y$ for all $n\in\w$. Therefore there exists a sequence
$\la L^\U_n:n\in\w\ra$ such that $L^\U_n\in [D^\U_n]^{<\w}$ and for every open non-empty
$V\subset Y$ we have $L^\U_n\cap V\neq\emptyset$ for all but finitely many $n$. 
For every $n\in\Omega_U$ find $K^\U_n\in [U]^{<\w}$ such that
for every $y\in L^\U_n$ there exists $x\in K^\U_n$
such that $\la x,y\ra\in D_n$, and set $R(\U)=\la K^\U_n:n\in\w\ra$. Note that
$R$ is such as in the definition of box-separability because $K^\U_n\subset U$
for all $n\in \Omega_U$ and the latter set is infinite. 
Since $X$ is box-separable there exists a family $\mathsf U$ of countable collections of 
open non-empty subsets of $X$ of size $|\mathsf U|=\w_1$, and 
such that for every open non-empty $U\subset X$ there exists $\U\in\mathsf U$
with the property $R(\U)(n)\subset U$ for infinitely many $n$. Since each $D_n$
is countable and $|\U|<\hot b$, there exists a sequence
$\la F_n:n\in\w\ra$ such that $F_n\in [D_n]^{<\w}$ and for every $\U\in\mathsf U$
we have $F_n\supset (K^\U_n\times L^\U_n)\cap D_n$ for all but finitely many 
$n\in\w$. 

We claim that $\bigcup_{n\in\w}F_n$ is dense in $X\times Y$.
Indeed, let us fix open non-empty subset of $X\times Y$ of the form $U\times V$
and find $\U\in\mathsf U$ with the property $R(\U)(n)=K^\U_n\subset U$ for infinitely many $n$, 
say for all $n\in I\in [\w]^\w$. Passing to a co-finite subset of $I$, 
we may assume if necessary, that $F_n\supset (K^\U_n\times L^\U_n)\cap D_n$ for all $n\in I$.
Finally, fix $n\in I$ such that $L^\U_n\cap V\neq\emptyset$ and pick $y\in L^\U_n\cap V$.
By the definition of $D^\U_n$ and $L^\U_n\subset D^\U_n$ we can find
$x\in K^\U_n$ such that $\la x,y\ra\in D_n$. Then $\la x,y\ra\in U\times V$
and $\la x,y\ra\in F_n$ because $\la x,y\ra\in K^\U_n\times L^\U_n$
and $\la x,y\ra\in D_n $. This completes our proof.
\end{proof}

\end{document}